\documentclass[reqno,tbtags,sumlimits,intlimits,a4paper,oneside,11pt]{amsart}
\def\x{2.4}
\usepackage[lmargin=\x cm,rmargin=\x cm,tmargin=\x cm,bmargin=\x cm,bindingoffset=0cm]{geometry}
\usepackage{amsthm}
\usepackage{nicefrac}
\usepackage{amsmath}
\usepackage{amssymb}
\usepackage{amsfonts}
\usepackage{mathrsfs}
\usepackage{tikz}
\usetikzlibrary{decorations.markings}
\usetikzlibrary{cd}
\usepackage{url}
\usepackage{color} 
\usepackage{blkarray}

\usepackage{mathscinet}

\newcounter{IssueCounter}
\newtheorem{Issue}[IssueCounter]{Issue}

\newcounter{QuestionCounter}
\newtheorem{Question}[QuestionCounter]{Question}

\definecolor{color2}{rgb}{0.4,0.2,0.0}
\newtheorem*{Answer*}{Answer}

\definecolor{color1}{rgb}{0.0,0.4,0.2}
\newcounter{ReminderCounter}
\newtheorem{Reminder}[ReminderCounter]{Reminder}

\newcommand{\mb}[1]{\mathbb{#1}}
\newcommand{\mf}[1]{\mathfrak{#1}}

\def\beq#1#2\eeq{%
        \begin{equation}%
        \label{#1}%
            #2%
        \end{equation}%
    }

\newcommand{\slnc}[1][n]{\mf{sl}({#1},\mb{C})}

\newcommand{\rg}{\finitegroupfont{\Gamma}}

\newcommand{\GL}{\mathrm{GL}}

\newcommand{\SO}{\mathrm{SO}}

\newcommand{\PSL}{\mathrm{PSL}}

\newcommand{\Aut}[1]{\mathrm{Aut}\!\left(#1 \right)}

\newcommand{\diag}{\mathrm{diag}}

\newtheorem{Theorem}{Theorem}[section]

\newtheorem{Lemma}[Theorem]{Lemma}

\newtheorem{Example}[Theorem]{Example}

\newcommand\colorchain{gray}
\newcommand\colorp{red}
\newcommand\colorm{blue}
\tikzstyle{root}=[scale=1,shape=circle ]
\tikzstyle{2ndroot}=[scale=0.8,shape=circle ]
\tikzstyle{cross}=[scale=0.15,shape=circle, fill]
\tikzstyle{chain}=[color=\colorchain,->,shorten >=8pt,shorten <=8pt, >=stealth]
\tikzstyle{schain}=[color=\colorchain,-,shorten >=4pt,shorten <=4pt]
\tikzstyle{cochainp}=[thick, color=\colorp,->,shorten >=8pt,shorten <=8pt, >=stealth]
\tikzstyle{cochainm}=[thick, color=\colorm,->,shorten >=8pt,shorten <=8pt, >=stealth]
\tikzstyle{scochainp}=[color=\colorp,-,shorten >=3pt,shorten <=3pt]
\tikzstyle{scochainm}=[color=\colorm, dashed,-,shorten >=3pt,shorten <=3pt]

\newcommand{\scaleGG}{1.3}

\newcommand{\dynkinfont}{\normalsize}
\tikzstyle{dynkinnode}=[draw, color=black, shape=circle,minimum size=3.5 pt,inner sep=0]
\tikzstyle{nodynkinnode}=[inner sep=-1]
\tikzstyle{ldynkinnode}=[draw, color=blue, shape=circle,minimum size=3.5 pt,inner sep=0]
\tikzstyle{sdynkinnode}=[draw, color=red, shape=circle,minimum size=3.5 pt,inner sep=0]
\tikzstyle{mdynkinnode}=[draw, color=magenta, shape=circle,minimum size=3.5 pt,inner sep=0]
\tikzstyle{fdynkinnode}=[draw, color=blue, shape=circle,minimum size=3.5 pt,inner sep=0,fill=black]
\tikzstyle{fldynkinnode}=[draw, color=blue, shape=circle,minimum size=3.5 pt,inner sep=0,fill=black]
\tikzstyle{fsdynkinnode}=[draw, color=red, shape=circle,minimum size=3.5 pt,inner sep=0,fill=black]
\tikzstyle{fmdynkinnode}=[draw, color=magenta, shape=circle,minimum size=3.5 pt,inner sep=0,fill=black]
\tikzstyle{brace}=[ decorate, decoration={brace, amplitude=5pt}]
\tikzstyle{mbrace}=[decorate, decoration={brace, amplitude=5pt, mirror}]

\newcommand{\Vmin}{V}
\newcommand{\groupg}{G_2(\mb C)}
\newcommand{\algg}{\mf g_2(\mb C)}

\newcommand{\short}{2}
\newcommand{\dyng}[2]{\begin{tikzpicture}[scale=0.6,baseline=(one.base), font=\dynkinfont,decoration={
    markings,
        mark=at position 0.5 with {
\draw (-3pt,-3pt) -- (0pt,0pt);
\draw (0pt,-0pt) -- (-3pt,3pt);
} }]
\path 
node at ( 0,0) [nodynkinnode,label=90: $$] (one) {$#2$}
node at ( 1,0) [nodynkinnode,label=90: $$] (two) {$#1$}
;
\draw[shorten <=\short pt,shorten >=\short pt] (one) to node []{$ $} (two);
\draw[shorten <=\short pt,shorten >=\short pt] (one.45) to node []{$ $} (two.135);
\draw[shorten <=\short pt,shorten >=\short pt] (one.-45) to node []{$ $} (two.-135);
\fill[postaction={decorate}] (two) -- (one);
\end{tikzpicture}}

\newcommand{\affdyng}[3]{\begin{tikzpicture}[scale=0.6,baseline=(zero.base), font=\dynkinfont,decoration={
    markings,
        mark=at position 0.5 with {
\draw (-3pt,-3pt) -- (0pt,0pt);
\draw (0pt,-0pt) -- (-3pt,3pt);
} }]
\path 
node at ( 2,0) [nodynkinnode,label=90: $$] (zero) {$#1$}
node at ( 0,0) [nodynkinnode,label=90: $$] (one) {$#3$}
node at ( 1,0) [nodynkinnode,label=90: $$] (two) {$#2$}
;
\draw[gray,shorten <=\short pt,shorten >=\short pt] (zero) to node []{$ $} (two);
\draw[shorten <=\short pt,shorten >=\short pt] (one) to node []{$ $} (two);
\draw[shorten <=\short pt,shorten >=\short pt] (one.45) to node []{$ $} (two.135);
\draw[shorten <=\short pt,shorten >=\short pt] (one.-45) to node []{$ $} (two.-135);
\fill[postaction={decorate}] (two) -- (one);
\end{tikzpicture}}

\newcommand{\finitegroupfont}{\mathsf}
\newcommand{\cg}[1]{\finitegroupfont{C}_{#1}}
\newcommand{\dg}[1]{\finitegroupfont{D}_{#1}}

\newcommand{\tg}{\finitegroupfont{T}}
\newcommand{\og}{\finitegroupfont{O}}
\newcommand{\ig}{\finitegroupfont{I}}

\title{
Polyhedral groups in $\groupg$
}
\author{Vincent Knibbeler, Sara Lombardo, Casper Oelen}

\begin{document}
\maketitle

\begin{abstract}
We classify embeddings of the finite groups $A_4$, $S_4$ and $A_5$ in the Lie group $\groupg$ up to conjugation.
\end{abstract}

\section{Introduction}
Cartan classified inner automorphisms of finite order of simple Lie algebras $\mf g$ over the complex numbers, up to conjugation
\cite{cartan1927geometrie}. See Reeder \cite{reeder2010torsion} for a modern recollection. 
Motivated by developments in the theory of automorphic Lie algebras (see \textit{e.g.} \cite{lombardo2005reduction,lombardo2010on,knibbeler2017higher,knibbeler2019hereditary,knibbeler2022automorphic}),
we would like to extend this classification of embeddings of cyclic groups to all finite subgroups of $\PSL(2,\mb C)$, which we will call polyhedral groups, consisting of cyclic groups of order $n$, dihedral groups of order $2n$, the tetrahedral group, the octahedral group, and the icosahedral group, denoted respectively
$$\cg{n},\quad \dg{n},\quad \tg,\quad\og,\quad\ig.$$
The groups $\tg$, $\og$ and $\ig$ are isomorphic to $A_4$, $S_4$ and $A_5$ respectively.

If $\mf{g}$ is one of the classical simple Lie algebras, one can find a satisfying classification using character theory for finite groups. For the exceptional Lie algebras, the full solution to this problem is still out of reach. The important case of embedding the icosahedral group and its double cover into $E_8({\mb C})$ was first solved by Frey \cite{frey1998conjugacy}. Recently, Frey and Rudelius \cite{frey20206d}  completed the classification of homomorphisms of (binary) polyhedral groups into $E_8({\mb C})$ and further refined its connection to 6-dimensional superconformal field theories. Moreover, they corrected some errors in the physics and mathematics literature on this topic, and thus (almost) reconciled their results.

 Throughout, let $\Vmin$ be the $7$-dimensional faithful representation of $\algg$. Its exponents generate a Lie group $\groupg$ in $\SO(V)$.
The centre of $\groupg$ is trivial and all automorphisms of $\algg$ are inner, which implies an isomorphism
$$\groupg\cong\Aut{\algg}$$
given by the adjoint representation. We will use this isomorphism without further mention and refer to Draper \cite{draper2018notes} for a comprehensive and accessible discussion of concrete models for groups of type $G_2$.

Finite subgroups of $\groupg$ have been classified by Cohen and Wales in 1983 \cite{cohen1983finite}, but the conjugacy classes of polyhedral groups are not listed in this paper. 
We are able to obtain this list for the groups $\tg$, $\og$ and $\ig$ with elementary methods because finite subgroups of $\groupg$ are conjugate if and only if they are conjugate in $\GL(V)$, the latter being decidable by character theory. This powerful theorem was obtained independently by Larsen in 1994 \cite{larsen1994on} and Griess in 1995 \cite{griess1995basic}.

Cohen and Griess \cite{cohen1987on} initiated a flurry of research into embeddings of simple and quasisimple groups, such as the icosahedral group $\ig$ and its double cover, into all exceptional Lie groups. See Frey and Ryba \cite{frey2018conjugacy} for a recent overview of the history and current state of the art. The classification of subgroups was finished in 2002 by Griess and Ryba \cite{griess2002classification},  who also settled the classification of conjugacy classes of embeddings in particular cases. They proved that there are precisely four conjugacy classes of monomorphisms $\ig\hookrightarrow\groupg$ which realise two conjugacy classes of icosahedral subgroups of $\groupg$ (the latter fact also obtained by Frey in 1998 \cite{frey1998bconjugacy}). 
 
In this paper, we classify monomorphisms from the tetrahedral and octahedral group into $\groupg$ up to conjugation and recover this classification for the icosahedral group obtained by Griess and Ryba with a different proof.
The classification of dihedral groups in $\groupg$ remains open.

\section{Cyclic groups in $\groupg$}
Elements of finite order in the connected Lie group $G=\Aut{\mf{g}}^0$ for a simple complex Lie algebra $\mf g$ are classified using the geometry of affine Weyl groups. The very short explanation is that any diagonalisable element of $G$ is conjugate to an element of a Cartan subgroup $T$, and two elements in a $T$ are conjugate in $G$ if and only if they are conjugate by the Weyl group $N_G(T)/T$. If an element of $G$ has also finite order, it is a rational point in a (compact) maximal torus in $G$. A maximal torus is isomorphic, through the exponential map, with a real Cartan subalgebra (CSA) up to translations by the co-weight lattice. Thus the conjugacy classes of elements of finite order in $G$ are identified with rational linear combinations of simple co-weights in the CSA, modulo the action of the Weyl group and the co-weight lattice. This latter group is known as the extended affine Weyl group.

One of the great insights of Cartan was to work modulo the affine Weyl group instead (the semidirect product of the Weyl group and the co-root lattice, rather than the co-weight lattice), which has a simplex as fundamental domain in the CSA, and handle the remaining symmetry using Dynkin diagrams. Determining the vertices of this simplex then yields
\begin{Theorem}[Cartan] Let $\mf{g}$ be a simple complex Lie algebra
and $\{\alpha_1,\ldots,\alpha_\ell\}$ a base for its root system with highest root $\sum_{i=1}^\ell a_i\alpha_i$. Set $a_0=1$.

Elements of order $n$ in $\Aut{\mf g}^0$, up to conjugation, are in one-to-one correspondence with sequences of nonnegative relative prime integers $\{s_0,\ldots,s_\ell\}$ such that
$$n=\sum_{i=0}^\ell a_i s_i,$$
up to symmetry of the affine Dynkin diagram.
The conjugacy class associated to $\{s_0,\ldots,s_\ell\}$ is represented by the automorphism sending the Chevalley generator $E_j$ of the Lie algebra to $\zeta^{s_j} E_j$, where $j=0,\ldots,\ell$ and $\zeta=\exp\frac{2\pi i}{n}$.
\end{Theorem} 
The sequence $\{s_0,\ldots,s_\ell\}$ lists the coordinates for the class of automorphisms.
For the full story we refer to the original work of Cartan \cite{cartan1927geometrie} and Kac (who extended the result to all automorphisms of finite order) \cite{kac1969automorphisms,kac1990infinite} and the enlightening treatment of Reeder \cite{reeder2010torsion}. 
See Bourbaki \cite{bourbaki2002lie} for a thorough study of (extended) affine Weyl groups.

\begin{Example} The affine Dynkin diagram of Lie type $G_2$ is given by
$$\begin{tikzpicture}[scale=0.6, font=\dynkinfont,decoration={
    markings,
        mark=at position 0.5 with {
\draw (-3pt,-3pt) -- (0pt,0pt);
\draw (0pt,-0pt) -- (-3pt,3pt);
} }]
\path 
node at ( 2,0) [dynkinnode,label=90: $1$] (zero) {$$}
node at ( 0,0) [dynkinnode,label=90: $3$] (one) {$$}
node at ( 1,0) [dynkinnode,label=90: $2$] (two) {$$}
;
\draw[gray] (zero) to node []{$ $} (two);
\draw[] (one) to node []{$ $} (two);
\draw[] (one.45) to node []{$ $} (two.135);
\draw[] (one.-45) to node []{$ $} (two.-135);
\fill[postaction={decorate}] (two) -- (one);
\end{tikzpicture}$$
with weights $a_i$ written above the nodes. If we look for automorphisms $g$ of order $3$ we find two conjugacy classes. With coordinates $s_i$ written in the diagram\footnote{
There is no canonical choice which simple root to call $\alpha_1$ and which to call $\alpha_2$. In the books of Bourbaki \cite{bourbaki2002lie} and Kac \cite{kac1990infinite} that are used by many, distinct choices are made. For this reason we decide to write the coordinates in the Dynkin diagram, so that the reader can keep their favourite book on the side without needing to translate.
}, they are $\affdyng{1}{1}{0}$ and $\affdyng{0}{0}{1}$.
The automorphism $g$ can be presented by extending these coordinates to the root system additively, modulo $n=3$, yielding a diagram of the eigenvalues of $g$ at the root spaces of $\algg$, cf.~Figure \ref{fig:root systems order 3}. The weights of the representation $\Vmin$ correspond to the short roots together with zero. Therefore, one can easily obtain the trace of $g$ on $\Vmin$ from the diagram in Figure \ref{fig:root systems order 3}.
\begin{center}
\begin{figure}[ht!]
\caption{Order $3$ elements in $\groupg$ diagonalised by a Cartan Weyl basis. Thick lines indicate simple roots. The number $s$ at a root indicates an eigenvalue $e^{\frac{2\pi i s}{3}}$ at the root space.} 
\label{fig:root systems order 3}
\begin{center}
\begin{tabular}{ccc} 
\\
$\affdyng{1}{1}{0}$&&$\affdyng{0}{0}{1}$\\
$
\begin{tikzpicture}[scale=\scaleGG]
\path 
node at ( 0,0) [root,label=270: $ $] (zero) {$ $}	
node at ( 0,0) [2ndroot,label=270: $ $] (zero) {$ $}
node at (-1.5,.866) [root,label=90: $ $] (m10) {$1 $}
node at (1,0) [root,label=0: $ $] (m01) {$0$}
node at (-.5,.866) [root,label=90: $ $] (m11) {$1 $}
node at (.5,.866) [root,label=90: $ $] (m12) {$1 $}
node at (1.5,.866) [root,label=90: $ $] (m13) {$1 $}
node at (0,2*.866) [root,label=90: $ $] (m23) {$2 $}
node at (1.5,-.866) [root,label=270: $ $] (n10) {$2 $}
node at (-1,0) [root,label=180: $ $] (n01) {$0 $}
node at (.5,-.866) [root,label=270: $ $] (n11) {$2 $}
node at (-.5,-.866) [root,label=270: $ $] (n12) {$2 $}
node at (-1.5,-.866) [root,label=270: $ $] (n13) {$2 $}
node at (0,-2*.866) [root,label=270: $ $] (n23) {$1 $}		
;
\draw[thick] (zero) to node []{$ $} (m10);
\draw[thick] (zero) to node []{$ $} (m01);
\draw[gray] (zero) to node []{$ $} (m11);
\draw[gray] (zero) to node []{$ $} (m12);
\draw[gray] (zero) to node []{$ $} (m13);
\draw[gray] (zero) to node []{$ $} (m23);
\draw[gray] (zero) to node []{$ $} (n10);
\draw[gray] (zero) to node []{$ $} (n01);
\draw[gray] (zero) to node []{$ $} (n11);
\draw[gray] (zero) to node []{$ $} (n12);
\draw[gray] (zero) to node []{$ $} (n13);
\draw[thick] (zero) to node []{$ $} (n23);
\end{tikzpicture}
$
&&
$
\begin{tikzpicture}[scale=\scaleGG]
\path 
node at ( 0,0) [root,label=270: $ $] (zero) {$ $}	
node at ( 0,0) [2ndroot,label=270: $ $] (zero) {$ $}
node at (-1.5,.866) [root,label=90: $ $] (m10) {$0 $}
node at (1,0) [root,label=0: $ $] (m01) {$1$}
node at (-.5,.866) [root,label=90: $ $] (m11) {$1 $}
node at (.5,.866) [root,label=90: $ $] (m12) {$2 $}
node at (1.5,.866) [root,label=90: $ $] (m13) {$0 $}
node at (0,2*.866) [root,label=90: $ $] (m23) {$0 $}
node at (1.5,-.866) [root,label=270: $ $] (n10) {$0 $}
node at (-1,0) [root,label=180: $ $] (n01) {$2 $}
node at (.5,-.866) [root,label=270: $ $] (n11) {$2 $}
node at (-.5,-.866) [root,label=270: $ $] (n12) {$1 $}
node at (-1.5,-.866) [root,label=270: $ $] (n13) {$0 $}
node at (0,-2*.866) [root,label=270: $ $] (n23) {$0 $}		
;
\draw[thick] (zero) to node []{$ $} (m10);
\draw[thick] (zero) to node []{$ $} (m01);
\draw[gray] (zero) to node []{$ $} (m11);
\draw[gray] (zero) to node []{$ $} (m12);
\draw[gray] (zero) to node []{$ $} (m13);
\draw[gray] (zero) to node []{$ $} (m23);
\draw[gray] (zero) to node []{$ $} (n10);
\draw[gray] (zero) to node []{$ $} (n01);
\draw[gray] (zero) to node []{$ $} (n11);
\draw[gray] (zero) to node []{$ $} (n12);
\draw[gray] (zero) to node []{$ $} (n13);
\draw[thick] (zero) to node []{$ $} (n23);
\end{tikzpicture}
$
\end{tabular}
\end{center}
\end{figure}
\end{center}

\end{Example}
We present all conjugacy classes of automorphisms of order $\le 5$ in Table \ref{tab:torsions}.
\begin{center}
\begin{table}[ht!] 
\caption{Elements of $\groupg$ of order $\le 5$ (with $\phi^{\pm}=\frac{1\pm\sqrt{5}}{2}$).}
\label{tab:torsions}
\begin{center}
\begin{tabular}{ccc} \hline
order&coordinates&trace on $\Vmin$
\\\hline
2&\affdyng{0}{1}{0}&$-1$
\\\hline
3&\affdyng{1}{1}{0}&$1$
\\&\affdyng{0}{0}{1}&$-2$
\\\hline
4&\affdyng{2}{1}{0}&$3$
\\&\affdyng{1}{0}{1}&$-1$
\\\hline
5&\affdyng{3}{1}{0}&$1+2\phi^+$
\\&\affdyng{1}{2}{0}&$1+2\phi^-$
\\&\affdyng{2}{0}{1}&$-\phi^{-}$
\\&\affdyng{0}{1}{1}&$-\phi^{+}$
\\\hline 
\end{tabular}
\end{center}
\end{table}
\end{center}

Two more lemmas are needed in preparation for the next section.
\begin{Lemma}
\label{lem:no PSL2C}
If $\PSL(2,\mb{C})\hookrightarrow\groupg$ is a monomorphism, then the conjugacy class of order $3$ elements in $\PSL(2,\mb{C})$ is mapped into the conjugacy class of $\affdyng{1}{1}{0}$ in $\groupg$.
\end{Lemma}
\begin{proof}
An element $g$ of order $3$ in $\PSL(2,\mb C)$ is conjugate to $\pm\diag(e^\frac{2\pi i}{6},e^{-\frac{2\pi i}{6}})=\pm\ e^{\frac{2\pi i}{6}H}$ where $H=\diag(1,-1)$ belongs to the Lie algebra $\slnc[2]$ of $\PSL(2,\mb C)$. Suppose $g$ is mapped to the conjugacy class of $\affdyng{0}{0}{1}$ in $\groupg$. Then $H$ acts on the simple root spaces with weights $\dyng{0}{2}$. By linear extension over the root system, we find all weights of $\algg$ as representation of $\slnc[2]$, and see weight $6$ has multiplicity two and weight $4$ has multiplicity one. Such a representation of $\slnc[2]$ does not exist. Hence, $g$ is mapped to the other conjugacy class of order $3$ elements in $\groupg$.
\end{proof}

\begin{Lemma}
\label{lem:character relation}
Let $g$ be a diagonalisable element of $\groupg$ with trace $\chi_\Vmin(g)$ and denote the trace of its action on $\algg$ by $\chi_{\algg}(g)$. Then $$\chi_{\algg}(g)=\frac{\chi_\Vmin(g)^2-\chi_\Vmin(g^2)-2\chi_\Vmin(g)}{2}.$$
\end{Lemma}
\begin{proof}
The Lie group $\groupg$ is realised as a Lie subgroup of $\SO(\Vmin)$. This turns $\mf{so}(\Vmin)$ into a $21$-dimensional representation of $\algg$, which has $\algg$ as $14$-dimensional subrepresentation. By complete reducibility there must be a $7$-dimensional representation $U$ such that  $\mf{so}(\Vmin)=\algg\oplus U$ as $\algg$-representation. It follows from the classification of representations of $\algg$ that $U$ is either trivial or $U=\Vmin$. If $U$ is trivial, then $\algg$ is a nontrivial ideal in $\mf{so}(\Vmin)$, contradicting the simplicity of the latter. Hence $\mf{so}(\Vmin)=\algg\oplus \Vmin.$ 

If we consider the trace of $g$ on the left- and right-hand side and observe that the trace of $g$ on $\mf{so}(\Vmin)$ is related to $\chi_\Vmin(g)$ by $\chi_{\mf{so}(\Vmin)}(g)=(\chi_\Vmin(g)^2-\chi_\Vmin(g^2))/2$ we obtain the desired result.
\end{proof}

\section{$\tg\og\ig$ groups in $\groupg$}
\newcommand{\ga}{\gamma_a}
\newcommand{\gb}{\gamma_b}
\newcommand{\gc}{\gamma_c}
The $\tg\og\ig$ groups have presentation
\[\rg=\langle\ga,\gb,\gc\,|\,\ga^n, \gb^3,\gc^2,\ga\gb\gc\rangle.\]
Taking $n=3, 4$ or $5$ results in  $\tg, \og$ or $\ig$ respectively.
\begin{center}
\begin{table}[ht!]
\caption{Irreducible characters of $\tg$, with $\zeta=e^{\frac{2\pi i}{3}}$, of $\og$ and $\ig$, with 
$\phi^{\pm}=\frac{1\pm\sqrt{5}}{2}$.} 
\label{tab:ctt}
\begin{center}
\small
\begin{minipage}{0.28\linewidth}
\vspace{-4mm}
\begin{tabular}{c|ccccccc} 
$\tg$& $[1]$ & $[\ga]$ &$[\gc]$&$[\gb]$\\[1mm]
\hline\\[-3.5mm]
$\chi_1$ & $1$ &$1$&$1$&$1$\\
$\chi_2$ & $1$ &$\zeta^2$&$1$&$\zeta$\\
$\chi_3$ & $1$ &$\zeta$&$1$&$\zeta^2$\\
$\chi_4$ & $3$ &$0$&$-1$&$0$\\
\end{tabular}
\end{minipage}
\begin{minipage}{0.35\linewidth}
\begin{tabular}{c|cccccccc}
 $\og$& $[1]$ & $[\gc]$ &$[\gb]$&$[\ga^2]$&$[\ga]$\\[1mm]
\hline\\[-3.5mm]
$\chi_1$ & $1$ &$1$&$1$&$1$&$1$\\
$\chi_2$ & $1$ &$-1$&$1$&$1$&$-1$\\
$\chi_3$ & $2$ &$0$&$-1$&$2$&$0$\\
$\chi_4$ & $3$ &$-1$&$0$&$-1$&$1$\\
$\chi_5$ & $3$ &$1$&$0$&$-1$&$-1$\\
\end{tabular}
\end{minipage}
\begin{minipage}{0.35\linewidth}
\begin{tabular}{c|cccccccc} 
 $\ig$& $[1]$ & $[\gb]$ &$[\gc]$&$[\ga]$&$[\ga^2]$\\[1mm]
\hline\\[-3.5mm]
$\chi_1$ & $1$ &$1$&$1$&$1$&$1$\\
$\chi_2$ & $3$ &$0$&$-1$&$\phi^-$&$\phi^+$\\
$\chi_3$ & $3$ &$0$&$-1$&$\phi^+$&$\phi^-$\\
$\chi_4$ & $4$ &$1$&$0$&$-1$&$-1$\\
$\chi_5$ & $5$ &$-1$&$1$&$0$&$0$\\
\end{tabular}
\normalsize
\end{minipage}
\end{center}
\end{table}
\end{center}

\begin{Lemma}
\label{lem:trace order 3}
If $\rg$ is a $\tg\og\ig$ group and $\rg\hookrightarrow\groupg$ a monomorphism, then any element $\gamma\in\rg$ of order $3$ is mapped to the class of $\affdyng{1}{1}{0}$ in $\groupg$.
\end{Lemma}
\begin{proof}In Table \ref{tab:torsions} we see that $\groupg$ only has one class of involutions, which has trace $-1$, and two classes of elements of order $3$, with traces $1$ (at the class of $\affdyng{1}{1}{0}$) and $-2$. The claim follows by observing that the $\tg\og\ig$ groups do not have a seven dimensional character with values $-1$ and $-2$ at the elements of order $2$ and $3$ respectively.
\end{proof}
In Table \ref{tab:characters}, we list all $7$-dimensional characters of $\tg\og\ig$ groups with value $-1$ and $1$ at elements of order $2$ and $3$ respectively, and irrational value at elements of order $5$. Due to Lemma \ref{lem:trace order 3} and Table \ref{tab:torsions}, we know that these conditions are necessary for the character of a monomorphism of a $\tg\og\ig$ group into $G_2(\mathbb{C})$.
\begin{center}
\begin{table}[ht!] 
\caption{Characters of $\tg\og\ig$ groups in $\groupg$.}
\label{tab:characters}
\begin{center}
\begin{tabular}{llllllll} 
$\tg$ && $\og$ && $\ig$\\
\hline\\[-3mm]
$\chi_1+2\chi_4$ && $\chi_1+2\chi_4$ && $\chi_1+2\chi_2$\\
&&$\chi_2+\chi_4+\chi_5$&&$\chi_1+2\chi_3$
\\&&&&$\chi_2+\chi_4$
\\&&&&$\chi_3+\chi_4$
\\[2mm]\hline 
\end{tabular}
\end{center}
\end{table}
\end{center}

One way to construct embeddings of polyhedral groups is through a composition $$\rg\hookrightarrow\PSL(2,\mb C)\hookrightarrow\groupg.$$
There is one embedding $\tg\hookrightarrow\PSL(2,\mb C)$, two embeddings $\og\hookrightarrow\PSL(2,\mb C)$ and also two embeddings $\ig\hookrightarrow\PSL(2,\mb C)$, up to conjugation.
Moreover, there are precisely two embeddings $\PSL(2,\mb C)\hookrightarrow \groupg$ up to conjugation \cite{collingwood1993nilpotent}. They are represented by the weighted Dynkin diagrams $\dyng{2}{2}$ and $\dyng{0}{2}$. The linear extension of these weights to the root system yields the weights of $\algg$ as a representation of $\slnc[2]$.

The characters of the various compositions realise all options from Table \ref{tab:characters}. From character theory we know that this classifies the embeddings in $\groupg$ up to conjugation in $\GL(\Vmin)$. Thanks to the work of Larsen and Griess \cite{larsen1994on,griess1995basic}, we can conclude that these conjugation classes correspond to the conjugation classes in $\groupg$. Thus we arrive at our main results.
\begin{Theorem}
\label{thm:main}
The conjugation classes of monomorphisms of $\tg$, $\og$ and $\ig$ into $\groupg$ are classified by the characters in Table \ref{tab:characters}.
\end{Theorem}

The part of this theorem concerning the icosahedral group can also be found in \cite[Section 4]{griess2002classification} where a different proof is given.

There is an automorphism of $\ig$ sending $\ga$ to $\ga^2$. 
The conjugacy classes of monomorphisms $\ig\hookrightarrow\groupg$ with the first two and last two icosahedral characters of Table \ref{tab:characters} are interchanged when precomposed with this automorphism. Thus, we see that there are only two conjugacy classes of images of monomorphisms $\ig\hookrightarrow\groupg$, and recover \cite[Theorem 4.11]{frey1998bconjugacy}.

\begin{Theorem}
\label{thm:PSL2C}
Each monomorphism of $\tg$, $\og$ and $\ig$ into $\groupg$ factors through $\PSL(2,\mb C)$.
\end{Theorem}
This result provides a practical construction, since monomorphisms $\rg\hookrightarrow\PSL(2,\mb C)$ and $\PSL(2,\mb C)\hookrightarrow\groupg$  can be found in the literature.
\begin{Theorem}
\label{thm:regular}
If $\rg$ is a $\tg\og\ig$ group embedded in $\Aut{\algg}$, then the only element in $\algg$ fixed by all $\gamma\in\rg$ is $0$.
\end{Theorem}
\begin{proof}
Using Lemma \ref{lem:character relation} and Theorem \ref{thm:main}, we can compute all characters of $\rg$-actions on $\algg$ and observe that none of them has a trivial component.  
\end{proof}


We have not classified embeddings of the dihedral groups in $\groupg$. The following example shows why this task cannot be completed with the same approach.

\begin{Example}
We construct a monomorphism $\dg{3}\hookrightarrow\groupg$ which shows that Lemma \ref{lem:trace order 3}, Theorem \ref{thm:PSL2C} and Theorem \ref{thm:regular} all fail for dihedral groups. 
We do so with a concrete model $L$ of $\algg$ in $\mf{gl}(7,\mb C)$. Let $x=(x_1, x_2, x_3)^t$ and let
\[
L=\left\{\begin{pmatrix}0&-\sqrt{2}y^t&-\sqrt{2}x^t\\\sqrt{2}x&a&l_y\\\sqrt{2}y&l_x&-a^t\end{pmatrix}\,|\,a\in\slnc[3],\,x,y\in\mb{C}^3\right\},
\quad l_{x}=\begin{pmatrix}0&-x_3&x_2\\x_3&0&-x_1\\-x_2&x_1&0\end{pmatrix}.
\]
See \cite{draper2018notes} for a proof that (a conjugate of) $L$ is indeed a simple Lie subalgebra of $\mf{gl}(7,\mb C)$ of type $G_2$, so that we can identify  $L$ with $\algg$ and $\Aut{L}$ with $\groupg$.

Conjugation with the diagonal matrix $\diag(1,\zeta,\zeta,\zeta,\zeta^2,\zeta^2,\zeta^2)$, $\zeta=e^{\frac{2\pi i}{3}}$, 
defines an automorphism of $\mf{gl}(7,\mb C)$ which preserves $L$. Let $r$ be its restriction to $L$. Then $r$ is an order $3$ automorphism of $L$. 
The elements of $L$ fixed by $r$ form a Lie subalgebra isomorphic to $\slnc[3]$, hence $r$ belongs to the conjugacy class of $\affdyng{0}{0}{1}$ in $\groupg$ (cf.~Figure \ref{fig:root systems order 3}).

The map $M\mapsto -M^t$ also defines an automorphism of $\mf{gl}(7,\mb C)$ which preserves $L$.
Let $s$ be its restriction to $L$, an automorphism of order $2$. Then $rs=sr^{-1}$, hence $r$ and $s$ generate a dihedral group of order $6$ in $\groupg$. 

Contrary to the case of the $\tg\og\ig$-groups, the map $\dg{3}\hookrightarrow\groupg$ we have constructed does not factor through  $\PSL(2,\mb C)$ because of Lemma \ref{lem:no PSL2C} and the fact that its image has nontrivial intersection with the conjugacy class of $\affdyng{0}{0}{1}$.
Moreover, we compute that the elements fixed by $\dg{3}$ form a subalgebra of $\algg$ isomorphic to $\slnc[2]$.
\end{Example}

\textbf{Acknowledgements} We are grateful to Jan Sanders for very helpful and stimulating discussions. We thank the anonymous reviewer for interesting historical remarks. 

\textbf{Funding} 
This work is supported by the Engineering and Physical Sciences Research Council (EPSRC): the work of SL and VK is supported by the grant EP/V048546/1; the work of CO is supported by the grant EP/W522569/1.


\begin{thebibliography}{10}

\bibitem{bourbaki2002lie}
Nicolas Bourbaki, \emph{Lie groups and {L}ie algebras. {C}hapters 4--6},
  Elements of Mathematics (Berlin), Springer-Verlag, Berlin, 2002, Translated
  from the 1968 French original by Andrew Pressley. \MR{1890629}

\bibitem{cartan1927geometrie}
\'Elie Cartan, \emph{La g\'eom\'etrie des groupes simples}, Ann. Mat. Pura
  Appl. \textbf{4} (1927), no.~1, 209--256. \MR{1553103}

\bibitem{cohen1987on}
Arjeh~M. Cohen and Robert~L. Griess, Jr., \emph{On finite simple subgroups of
  the complex {L}ie group of type {$E_8$}}, The {A}rcata {C}onference on
  {R}epresentations of {F}inite {G}roups ({A}rcata, {C}alif., 1986), Proc.
  Sympos. Pure Math., vol.~47, Amer. Math. Soc., Providence, RI, 1987,
  pp.~367--405. \MR{933426}

\bibitem{cohen1983finite}
Arjeh~M. Cohen and David~B. Wales, \emph{Finite subgroups of {$G_{2}({\bf
  C})$}}, Comm. Algebra \textbf{11} (1983), no.~4, 441--459. \MR{689418}

\bibitem{collingwood1993nilpotent}
David~H. Collingwood and William~M. McGovern, \emph{Nilpotent orbits in
  semisimple {L}ie algebras}, Van Nostrand Reinhold Mathematics Series, Van
  Nostrand Reinhold Co., New York, 1993. \MR{1251060}

\bibitem{draper2018notes}
Cristina Draper~Fontanals, \emph{Notes on {$G_2$}: the {L}ie algebra and the
  {L}ie group}, Differential Geom. Appl. \textbf{57} (2018), 23--74.
  \MR{3758361}

\bibitem{frey2018conjugacy}
Darrin Frey and Alex Ryba, \emph{Conjugacy of embeddings of alternating groups
  in exceptional {L}ie groups}, Bull. Inst. Math. Acad. Sin. (N.S.) \textbf{13}
  (2018), no.~4, 463--480. \MR{3888882}

\bibitem{frey1998conjugacy}
Darrin~D. Frey, \emph{{C}onjugacy of {$\mathrm {Alt}_5$} and {$\mathrm {SL}(2,
  5)$} {S}ubgroups of {$E_8(\mathbb C)$}}, Memoirs of the American Mathematical
  Society, vol. 634, Amer. Math. Soc., Providence, RI, 1998.

\bibitem{frey1998bconjugacy}
\bysame, \emph{Conjugacy of {${\rm Alt}_5$} and {${\rm SL}(2,5)$} subgroups of
  {$E_6({\mb C})$}, {$F_4({\mb C})$}, and a subgroup of {$E_8({\mb C})$} of
  type {$A_2E_6$}}, J. Algebra \textbf{202} (1998), no.~2, 414--454.
  \MR{1617620}

\bibitem{frey20206d}
Darrin~D. Frey and Tom Rudelius, \emph{6{D} {SCFT}s and the classification of
  homomorphisms {$\Gamma_{ADE}\to E_8$}}, Adv. Theor. Math. Phys. \textbf{24}
  (2020), no.~3, 709--756. \MR{4145058}

\bibitem{griess1995basic}
Robert~L. Griess, Jr., \emph{Basic conjugacy theorems for {$G_2$}}, Invent.
  Math. \textbf{121} (1995), no.~2, 257--277. \MR{1346206}

\bibitem{griess2002classification}
Robert~L. Griess, Jr. and A.~J.~E. Ryba, \emph{Classification of finite
  quasisimple groups which embed in exceptional algebraic groups}, J. Group
  Theory \textbf{5} (2002), no.~1, 1--39. \MR{1879514}

\bibitem{kac1969automorphisms}
Victor~G. Kac, \emph{Automorphisms of finite order of semisimple {L}ie
  algebras}, Funkcional. Anal. i Prilo\v zen. \textbf{3} (1969), no.~3, 94--96.
  \MR{0251091}

\bibitem{kac1990infinite}
\bysame, \emph{Infinite-dimensional {L}ie algebras}, third ed., Cambridge
  University Press, Cambridge, 1990. \MR{1104219}

\bibitem{knibbeler2017higher}
Vincent Knibbeler, Sara Lombardo, and Jan~A. Sanders, \emph{Higher-dimensional
  automorphic {L}ie algebras}, Found. Comput. Math. \textbf{17} (2017), no.~4,
  987--1035. \MR{3682219}

\bibitem{knibbeler2019hereditary}
\bysame, \emph{Hereditary {A}utomorphic {L}ie {A}lgebras}, Commun. Contemp.
  Math. (2019).

\bibitem{knibbeler2022automorphic}
Vincent Knibbeler, Sara Lombardo, and Alexander~P Veselov, \emph{Automorphic
  {L}ie algebras and {M}odular {F}orms}, Int. Math. Res. Not. IMRN (2022),
  rnab376.

\bibitem{larsen1994on}
Michael Larsen, \emph{On the conjugacy of element-conjugate homomorphisms},
  Israel J. Math. \textbf{88} (1994), no.~1-3, 253--277. \MR{1303498}

\bibitem{lombardo2005reduction}
Sara Lombardo and Alexander~V. Mikhailov, \emph{Reduction groups and
  automorphic {L}ie algebras}, Comm. Math. Phys. \textbf{258} (2005), no.~1,
  179--202. \MR{2166845}

\bibitem{lombardo2010on}
Sara Lombardo and Jan~A. Sanders, \emph{On the classification of automorphic
  {L}ie algebras}, Comm. Math. Phys. \textbf{299} (2010), no.~3, 793--824.
  \MR{2718933}

\bibitem{reeder2010torsion}
Mark Reeder, \emph{Torsion automorphisms of simple {L}ie algebras}, Enseign.
  Math. (2) \textbf{56} (2010), no.~1-2, 3--47. \MR{2674853}

\end{thebibliography}

\def\cprime{$'$}
\providecommand{\bysame}{\leavevmode\hbox to3em{\hrulefill}\thinspace}
\providecommand{\MR}{\relax\ifhmode\unskip\space\fi MR }
\providecommand{\MRhref}[2]{%
  \href{http://www.ams.org/mathscinet-getitem?mr=#1}{#2}
}
\providecommand{\href}[2]{#2}

\end{document}